\documentclass[reqno]{amsart}
\usepackage{amssymb,amscd,graphicx,stmaryrd,mathrsfs,amsthm,mathtools}
\usepackage[all]{xy}
\usepackage{color}
\usepackage{url}
\usepackage
{hyperref}

\newtheorem{theorem}{Theorem}[section]
\newtheorem{proposition}[theorem]{Proposition} 
\newtheorem{corollary}[theorem]{Corollary}
\newtheorem{lemma}[theorem]{Lemma}
\newtheorem{remark}[theorem]{Remark}
\newtheorem{remarks}[theorem]{Remarks}
\newtheorem{remark&definition}[theorem]{Remark and Definition}
\newtheorem{question}[theorem]{Question}

\newtheorem{definition}[theorem]{Definition}

\makeatletter
  
  \@addtoreset{equation}{section}
\makeatother

\begin{document}

\title[Lebesgue decomposition for positive operators]{Lebesgue decomposition for positive operators revisited}
\author{Yoshiki Aibara and Yoshimichi Ueda}
\address{
Graduate School of Mathematics, Nagoya University, 
Furocho, Chikusaku, Nagoya, 464-8602, Japan
}
\email{
(YA) y.aibara.math95@gmail.com; 
(YU) ueda@math.nagoya-u.ac.jp
}
\date{\today}
\thanks{This work was supported in part by Grant-in-Aid for Scientific Research (B) JP18H01122.}
\begin{abstract} 
We explain how Pusz--Woronowicz's idea of their functional calculus fits the theory of Lebesgue decomposition for positive operators on Hilbert spaces initially developed by Ando. In this way, we reconstruct the essential and fundamental part of the theory.  
\end{abstract} 

\maketitle

\section{Introduction} 

Lebesgue decomposition for positive operators on Hilbert spaces was introduced by Ando \cite{Ando76} in the mid 70s, and then has been studied so far by many hands. It should be regarded as an attempt to generalize the Douglas decomposition theorem (see e.g., \cite[Theorem 2.1]{FW71}) and has a strong connection to the notion of parallel sum, which was originally motivated from electronic networks but has played an important role in the study of operator ranges illustrated by Fillmore and Williams \cite{FW71} as well as in the study of binary operations such as operator connections/means. The aim of these notes is to give a concise exposition on our idea towards yet another approach to Lebesgue decomposition, which is, in some sense, a reorganization of various ideas implicitly appeared in previous works by means of Pusz--Woronowicz functional calculus \cite{PW75} (see \cite{HU21,HUW22} for its operator-oriented treatment). We hope that the idea will further be investigated by specialists in the direction. 

Let $A, B$ be positive bounded operators on a Hilbert space $\mathcal{H}$. Ando introduced the concept that $A$ and $B$ are mutually singular, say $A \perp B$, and the one that $B$ is $A$-absolutely continuous, say $B \lessdot A$. He then introduced the maximal $A$-absolutely continuous part $[A]B$, and of course, $[A]B \perp (B-[A]B)$ holds true. Our initial observation is that $(A,B) \mapsto [A]B$ is a binary operation and can be understood as a very special case of Pusz--Woronowicz functional calculus; see Theorem \ref{T3.8}. This observation was implicitly appeared in \cite{HUW22}, but has not systematically be investigated so far. The observation explains that Pusz--Woronowicz's original construction of their functional calculus is essentially the same as von Neumann's famous clever proof (\cite[Lemma 3.2.3]{vN40-1} for its original form; \cite[Theorem 6.10]{Rudin} for its exposition) of Lebesgue--Radon--Nikodym theorem for measures based on Hilbert space techniques. See Remark \ref{R3.4}. Moreover, it also enables us to investigate the maximal $A$-absolutely continuous part $[A]B$ as Pusz--Woronowicz functional calculus. For example, it is rather straightforward in our approach to prove Kosaki's formula (that is now available in \cite{Ko})
\[
[A_1\otimes A_2](B_1\otimes B_2) = [A_1]B_1\otimes[A_2]B_2.
\] 
See Corollary \ref{C3.9}. Moreover, the mechanism behind this formula becomes completely clear in our approach. 

The previous works \cite{HU21},\cite{HUW22} and this work altogether show that the Pusz--Woronowicz functional calculus is a right framework to discuss all the well-established binary operations for positive operators, operator means/connections, operator perspectives and the maximal absolutely continuous part, in a unified way. Moreover, all those binary operations are special cases of the Pusz--Woronowicz functional calculus. 

Part of the contents of these notes are originally conducted as part of the first author's thesis project under the second author's guidance. 

\medskip\noindent
{\bf Notation.} In these notes, $\mathcal{H}$, $\mathcal{H}_i$ and $\mathcal{K}$ denote Hilbert spaces, $B(\mathcal{H})$ all the bounded operators on $\mathcal{H}$, and $B(\mathcal{H})_*^+$ all the positive normal linear functionals on $B(\mathcal{H})$, i.e., $X \in B(\mathcal{H}) \mapsto \mathrm{Tr}(DX)\in \mathbb{C}$ with positive trace-class operator $D$ on $\mathcal{H}$, where $\mathrm{Tr}$ is the usual trace. The kernel and the range of an operator $A$ are denoted by $\mathrm{ker}(A)$ and $\mathrm{ran}(A)$, respectively. The closure of $\mathrm{ran}(A)$ is denoted by $\overline{\mathrm{ran}}(A)$.  

\medskip\noindent
{\bf Acknowledgements.} We would like to acknowledge Fumio Hiai for letting us know Kosaki's formula and also for his comments to a draft version of these notes. We would also like to acknowledge Hideki Kosaki for showing us a draft version of \cite{Ko}, including many interesting facts based on unbounded operator technologies, after sending him a draft version of these notes.   

\section{Pusz--Woronowicz functional calculus}

Let $\phi(x,y)$ be a real-valued homogeneous Borel function over $[0,\infty)^2$ that is bounded 
on any compact subset. We will use the formulation of \emph{Pusz--Woronowicz functional calculus} (\emph{PW-functional calculus} for short) $\phi(A,B)$ with two positive bounded operators $A,B$ on a Hilbert space given in \cite[Definition 4.1]{HUW22} (that originates in \cite[Remark 10]{HU21}), that is, a unique extension of the usual functional calculus for commuting pairs of bounded positive operators in such a way that the following property (called \emph{operator homogeneity}) holds: If a bounded operator $C : \mathcal{K} \to \mathcal{H}$ satisfies $\overline{\mathrm{ran}}((A+B)^{1/2}) \subseteq \overline{\mathrm{ran}}(C)$, then 
\[
\phi(C^* AC, C^* BC) = C^* \phi(A,B)C. 
\]
(We need the notion of extended lower semibounded self-adjoint part $\widehat{B(\mathcal{H})}_\mathrm{lb}$ when $\phi(x,y)$ is just assumed to be bounded from below on each compact subset of $[0,\infty)^2$; see \cite{HUW22}.) 

\medskip
We will explain an explicit construction for a not necessarily commuting pair $A,B$ of positive bounded operators on a Hilbert space $\mathcal{H}$. This construction itself is important in our approach to Lebesgue decomposition for positive operators. 

Write $\mathcal{H}_{A,B}\coloneqq\mathrm{ker}(A+B)^\perp = \mathrm{ker}((A+B)^{1/2})^\perp = \overline{\mathrm{ran}}((A+B)^{1/2})$, and define $T_{A,B} : \mathcal{H} \to \mathcal{H}_{A,B}$ by $\xi \in \mathcal{H} \mapsto (A+B)^{1/2}\xi \in \mathcal{H}_{A,B}$, which has a dense range by definition. The mappings 
\[
\left(A + B\right)^{1/2}\xi \mapsto A^{1/2} \xi, \quad \left(A + B\right)^{1/2}\xi \mapsto B^{1/2} \xi \quad  (\xi \in \mathcal{H})
\]
uniquely extend to contractive operators $X_{A,B}, Y_{A,B}$ from $\mathcal{H}_{A,B}$ to $\mathcal{H}$ such that $X_{A,B} T_{A,B} = A^{1/2}$, $Y_{A,B} T_{A,B} = B^{1/2}$.
We define two bounded operators $R_{A,B},S_{A,B}$ on $\mathcal{H}_{A,B}$ by $R_{A,B} \coloneqq |X_{A,B}|^2 = X_{A,B}^* X_{A,B}$ and $S_{A,B} \coloneqq |Y_{A,B}|^2 = Y_{A,B}^* Y_{A,B}$.
Let $X_{A,B} = U_{A,B} |X_{A,B}|, Y_{A,B} = V_{A,B}|Y_{A,B}|$ be the polar decompositions of $X_{A,B}, Y_{A,B}$, respectively.

\begin{lemma}[{\cite[Theorem 1.1]{PW75}; see also \cite[Lemma 1]{HU21}}] \label{L2.1}
$R_{A,B} + S_{A,B} = 1_{\mathcal{H}_{A,B}}$. In particular, $(R_{A,B},S_{A,B})$ is a commuting pair of positive bounded operators on $\mathcal{H}_{A,B}$. 
\end{lemma}

\begin{lemma}\label{L2.2}
Denote $\mathfrak{I}_{A, B} \coloneqq \{C \in B(\mathcal{H})\:;\: O \leq C \leq A + B\}$ and $\mathfrak{J}_{A,B} \coloneqq \{\tilde{C} \in B(\mathcal{H}_{A,B})\:;\: O \leq \tilde{C} \leq I_{\mathcal{H}_{A,B}}\}$.
The map $\Gamma_{A,B} : \mathfrak{J}_{A,B} \to \mathfrak{I}_{A,B}$ defined by
\[
\Gamma_{A,B}(\tilde{C}) = T_{A,B}^*\,\tilde{C}\,T_{A,B}
\]
is bijective and satisfies $\Gamma_{A,B}(R_{A,B}) = A$, $\Gamma_{A,B}(S_{A,B}) = B$ and $\Gamma_{A,B}(I_{\mathcal{H}_{A,B}}) = A + B$. Moreover, both $\Gamma_{A,B}$ and its inverse map $\Gamma_{A,B}^{-1}$ are order-preserving.
\end{lemma}
\begin{proof}
(Injectivity) If $\tilde{C}, \tilde{D} \in \mathfrak{J}_{A,B}$ are such that $\Gamma_{A,B}(\tilde{C}) = \Gamma_{A,B}(\tilde{D})$, then  
\[
((A + B)^{1/2}\xi,\tilde{C}(A + B)^{1/2}\eta)
= (\xi,\Gamma_{A,B}(\tilde{C})\eta) 
= (\xi,\Gamma_{A,B}(\tilde{D})\eta) 
= ((A + B)^{1/2}\xi, \tilde{C}(A + B)^{1/2}\eta) 
\]
for any $\xi, \eta \in \mathcal{H}$, and hence $\tilde{C} = \tilde{D}$.

(Surjectivity) Let $C \in \mathfrak{I}_{A,B}$ be arbitrarily chosen. Due to the Douglas decomposition theorem (see e.g., \cite[Theorem 2.1]{FW71}) there exists a unique $D \in B(\mathcal{H})$ such that $(A + B)^{1/2}D = C^{1/2}$ and $\Vert D\Vert \leq 1$ hold. Multiplying the support projection of $(A+B)^{1/2}$ to $D$ from the left, we may and do regard $D$ as a bounded operator from $\mathcal{H}$ to $\mathcal{H}_{A,B}$, and hence we have $D^* T_{A,B} = C^{1/2}$ by taking the adjoint. Letting $\tilde{C} = DD^{*} \in \mathfrak{J}_{A,B}$ we have $C = \Gamma_{A,B}(\tilde{C})$. 

Finally, $\Gamma_{A,B}$ is order-preserving by its construction, and hence so is the inverse map $\Gamma_{A,B}^{-1}$ trivially. 
\end{proof}

\begin{remarks}\label{R2.3} {\rm 
(1) The above map $\Gamma_{A,B}$ naturally extends to an order-preserving bijection from $\{C \in B(\mathcal{H})_{+}\:;\: O \leq C \leq \alpha(A + B)\ \text{for some $\alpha>0$}\}$ to $\{\tilde{C} \in B(\mathcal{H}_{A,B})_{+}\:;\: O \leq \tilde{C} \leq \beta I_{\mathcal{H}_{A,B}}\ \text{for some $\beta>0$}\}$. We still denote it by the same symbol.

(2) It is clear that $\Gamma_{A,B}$ is strongly continuous. Moreover, if $C_\lambda \nearrow C$ in $\mathfrak{I}_{A,B}$, then
\begin{align*}
(\Gamma_{A,B}^{-1}(C_\lambda)(A + B)^{1/2}\xi,(A + B)^{1/2}\xi)
= (C_\lambda\xi,\xi)
\nearrow (C\xi, \xi)
= ( \Gamma_{A,B}^{-1}(C)(A+B)^{1/2}\xi, (A + B)^{1/2}\xi) 
\end{align*}
for any $\xi \in \mathcal{H}$. Hence we have $\Gamma_{A,B}^{-1}(C_\lambda) \nearrow \Gamma_{A,B}^{-1}(C)$ because $\mathcal{H}_{A,B} = \overline{\mathrm{ran}}((A + B)^{1/2})$.}
\end{remarks}

With these preparations, we have
\[
\phi(A,B) 
= 
\phi(T_{A,B}^* R_{A,B} T_{A,B}, T_{A,B}^* S_{A,B} T_{A,B}) 
=
T_{A,B}^*\phi(R_{A,B}, S_{A,B})T_{A,B} 
=
\Gamma_{A,B}(\phi(R_{A,B},S_{A,B})) 
\]
by the operator homogeneity, and $\phi(R_{A,B}, S_{A,B})$ is the usual functional calculus by Lemma \ref{L2.1}. This is an explicit construction of $\phi(A,B)$ and actually, a translation of Pusz--Woronowicz's original construction of their functional calculus for positive sesquilinear forms \cite{PW75} in terms of Hilbert space operators.

\medskip
Here is a simple property on PW-functional calculus. We believe that this property has not been pointed out so far. 

\begin{theorem}\label{T2.4} 
If $\phi(x_1 x_2, y_1 y_2) = \phi(x_1,y_1)\phi(x_2,y_2)$ holds for every $(x_1,y_1),(x_2,y_2) \in [0,\infty)^2$, then 
\[
\phi(A_1\otimes A_2, B_1\otimes B_2) = \phi(A_1,B_1)\otimes\phi(A_2,B_2)
\]
holds for any $(A_i,B_i) \in B(\mathcal{H}_i)_+ \times B(\mathcal{H}_i)_+$, $i=1,2$. 
\end{theorem}
\begin{proof}
We set $T \coloneqq T_{A_1,B_1}\otimes T_{A_2,B_2}$ and $R \coloneqq R_{A_1,B_1}\otimes R_{A_2,B_2}$, $S \coloneqq S_{A_1,B_1}\otimes S_{A_2,B_2}$. Let
\[
R_{A_i,B_i} = \int_0^1 x\,E_i(dx), \quad (i=1,2)
\]
be the spectral decomposition. 

Since $T$ has a dense range in $\mathcal{H}_{A_1,B_1}\otimes\mathcal{H}_{A_2,B_2}$ and $A_1\otimes A_2 = T^* R T$, $B_1\otimes B_2 = T^* S T$, the operator homogeneity says that   
\[
\phi(A_1\otimes A_2, B_1\otimes B_2) 
= 
T^*\phi(R,S)T.
\]
By assumption, we have
\begin{align*}
\phi(R,S) 
&= 
\int_0^1 \int_0^1 \phi(x_1 x_2, (1-x_1)(1-x_2))\, E_1(dx_1)\otimes E_2(dx_2) \\
&=
\int_0^1 \int_0^1 \phi(x_1,1-x_1) \phi(x_2,1-x_2)\, E_1(dx_1)\otimes E_2(dx_2) \\
&=
\phi(R_{A_1,B_1},S_{A_1,B_1})\otimes\phi(R_{A_2,B_2},S_{A_2,B_2})  
\end{align*}
(see e.g., \cite[Theorem 8.2]{St59}), and hence we obtain the desired formula. 
\end{proof}

A typical example of $\phi(x,y)$ we can apply the above proposition to is $(x,y) \mapsto x^\alpha y^{1-\alpha} = (x/y)^\alpha\,y$. In particular, this result is applicable to the weighted mean $A\,\sharp_\alpha B$. 

\begin{remark}\label{R2.5} {\rm The above theorem is still valid in some sense even when $\phi(x,y)$ is just bounded from below on any compact subset, though its proof needs careful treatment of `tensor products' of extended lower semibounded parts. Assume $\psi(x,y) = x\log(x/y)$. Then, moreover, the way of the above proof naturally suggests that $\psi(A_1\otimes A_2, B_1\otimes B_2) = \psi(A_1,B_1)\otimes A_2 + A_1\otimes \psi(A_2,B_2)$ holds. This will be discussed in appendix A.} 
\end{remark}  

\medskip
In the rest of this section we devote to proving technical lemmas, which are necessary in the next section. 

\begin{lemma}\label{L2.6}
If $\phi(x,y)$ admits a representation $\phi(x,y) = yf(y)$ on $\{(x,y) \in [0,\infty)^2\,;\,x+y=1\}$ with some bounded Borel function $f(y)$ on $[0,1]$. Then we have
\[
\phi(A,B) = B^{1/2} V_{A,B} f(S_{A,B}) V_{A,B}^* B^{1/2}. 
\]
Moreover, if $f(y)$ is an indicator function, then $V_{A,B} f(S_{A,B}) V_{A,B}^*$ must be a projection.
\end{lemma}

\begin{proof}
Since $R_{A,B} + S_{A,B} = I_{\mathcal{H}_{A,B}}$, we have $\phi(R_{A,B},S_{A,B}) = S_{A,B}f(S_{A,B}) = S_{A,B}^{1/2} f(S_{A,B}) S_{A,B}^{1/2}$. Since $S_{A,B}^{1/2} = |Y_{A,B}| = V_{A,B}^* Y_{A,B} = Y_{A,B}^* V_{A,B}$ and $Y_{A,B}\,T_{A,B} = B^{1/2}$, we obtain 
\begin{align*}
\phi(A,B)
&= 
T_{A,B}^* S_{A,B}^{1/2}f(S_{A,B})S_{A,B}^{1/2}T_{A,B} \\
&= 
T_{A,B}^* Y_{A,B}^* V_{A,B} f(S_{A,B}) V_{A,B}^* Y_{A,B}T_{A,B} \\
&= 
B^{1/2} V_{A,B} f(S_{A,B}) V_{A,B}^* B^{1/2}.
\end{align*}
    
If $f(y)$ is an indicator function, the $f(S_{A,B})$ is a projection. Since $V_{A,B}^* V_{A,B}$ commutes with $f(S_{A,B})$, it follows that 
\[
(V_{A,B} f(S_{A,B}) V_{A,B}^*)^2 
= 
V_{A,B} f(S_{A,B}) V_{A,B}^* V_{A,B} f(S_{A,B}) V_{A,B}^*
= 
V_{A,B} f(S_{A,B}) V_{A,B}^*.
\]
Hence $V_{A,B} f(S_{A,B})V_{A,B}^*$ is a projection.
\end{proof}

\begin{lemma}\label{L2.7}
Let $f(y)$ be a bounded Borel function on $[0,1]$. Then we have
\begin{equation}\label{Eq2.1}
f(Y_{A,B}Y_{A,B}^*) = f(0)(I - V_{A,B} V_{A,B}^*) + V_{A,B} f(S_{A,B})V_{A,B}^*,         
\end{equation}
and moreover, 
\[
B^{1/2} V_{A,B} f(S_{A,B})V_{A,B}^* B^{1/2}
= 
B^{1/2}f(Y_{A,B}Y_{A,B}^*) B^{1/2}.
\]
\end{lemma}
\begin{proof}
Since $V_{A,B} S_{A,B}^n V_{A,B}^* = (V_{A,B} S_{A,B} V_{A,B}^*)^n = (Y_{A,B}Y_{A,B}^*)^n$ for all positive integers $n$, we have formula \eqref{Eq2.1} when $f(y)$ is a polynomial. Taking a uniform approximation to a given continuous function on $[0,1]$ by polynomials, we see that the same formula holds true even for any continuous function on $[0,1]$. Then, appealing to the monotone class theorem we finally confirm that formula \eqref{Eq2.1} holds for a general Borel function.

Let $Q = I - V_{A,B}V_{A,B}^{*}$ be the orthogonal projection onto $\ker(Y_{A,B}^*) = \overline{\rm{ran}}(Y_{A,B})^\perp$. We observe that 
\begin{align*}
\xi \in \ker(Y_{A,B}^*) 
&\Leftrightarrow (Y_{A,B}^*\xi, T_{A,B}\eta) = 0 \quad (\text{for all $\eta \in \mathcal{H}$}) \\
&\Leftrightarrow (\xi, B^{1/2}\eta) = 0 \quad (\text{for all $\eta \in \mathcal{H}$}) \\
&\Leftrightarrow (B^{1/2}\xi, \eta) = 0 \quad (\text{for all $\eta \in \mathcal{H}$}) \\
&\Leftrightarrow \xi \in \ker(B^{1/2}).
\end{align*}
Thus we obtain $B^{1/2}QB^{1/2} = O$. Moreover, we have
\[
B^{1/2}f(Y_{A,B}Y_{A,B}^{*})B^{1/2}
= 
B^{1/2}(f(0)Q + V_{A,B} f(S_{A,B})V_{A,B}^*)B^{1/2}
= 
B^{1/2}(V_{A,B} f(S_{A,B})V_{A,B}^*)B^{1/2}.
\]
Hence we are done.
\end{proof}

\begin{remark}\label{R2.8} {\rm 
Similarly, for any bounded Borel function $f(x)$ on $[0,1]$ we have
\begin{gather*}
f(X_{A,B}X_{A,B}^{*}) =f(0)(I - U_{A,B} U_{A,B}^*) + U_{A,B} f(R_{A,B})U_{A,B}^*, \\
A^{1/2}U_{A,B} f(R_{A,B})U_{A,B}^* A^{1/2} = A^{1/2}f(X_{A,B}X_{A,B}^*)A^{1/2}.
\end{gather*}}
\end{remark}

\section{Lebesgue decomposition of positive operators}

Let us begin by recalling the definitions introduced by Ando \cite{Ando76}. 

\begin{definition}\label{D3.1}
Let $A, B$ be positive bounded operators on a Hilbert space. We say that $A$ and $B$ are \textit{mutually singular}, denoted by $A \perp B$, if there is no non-zero bounded operator $C$ on the Hilbert space satisfying that both $O \leq C \leq A$ and $O \leq C \leq B$. We also say that $B$ is \textit{$A$-absolutely continuous}, denoted by $B \lessdot  A$, if there exists a sequence $B_n$ of positive bounded operators on the Hilbert space such that $B_{n} \nearrow B$ as $n \to \infty$ and $B_n \leq \alpha_n A$ for some $\alpha_n > 0$.
\end{definition}

In what follows, let $A, B$ be positive bounded operators on a Hilbert space $\mathcal{H}$. We will freely use the notation given in section 2.

\begin{lemma}\label{L3.2}
The following conditions are equivalent: 
\begin{itemize}
\item[(1)] $A \perp B$.
\item[(2)] $R_{A,B} \perp S_{A,B}$.
\item[(3)] $R_{A,B}$ and $S_{A,B}$ form an orthogonal pair of  projections.
\end{itemize}
\end{lemma}
\begin{proof}
In the proof, we will write $\Gamma = \Gamma_{A,B}$, $R = R_{A,B}$ and $S=S_{A,B}$ for simplicity. 

(1) $\Rightarrow$ (2):  
Assume that $O \leq C \leq R$ and $O \leq C \leq S$. Since $\Gamma(R) = A$ and $\Gamma(S) = B$ we have $O \leq \Gamma(C) \leq A$ and $O \leq \Gamma(C) \leq B$. By assumption $\Gamma(C)$ must be $O$, and hence $C = O$ too by Lemma \ref{L2.2}.

\medskip
(2) $\Rightarrow$ (3): Let $E_R$ be the spectral (projection-valued) measure of $R$. For any $0<\varepsilon<1$ we have
\[
O \leq \varepsilon E_R([\varepsilon, 1 - \varepsilon]) \leq R,\qquad 
O \leq \varepsilon E_R([\varepsilon, 1 - \varepsilon]) \leq I_{\mathcal{H}_{A,B}} - R = S. 
\]
By assumption we obtain $E_R([\varepsilon, 1 - \varepsilon]) = O$. Taking the limit as $n\to\infty$ with $\varepsilon = 1/n$ we have $E_R((0,1)) = O$ and $I_{\mathcal{H}_{A,B}} = E_R([0,1]) =E_R(\{0\}) + E_R(\{1\})$. Hence $R = (E_R(\{0\}) + E_R(\{1\}))R = E_R(\{1\})$, and similarly $S = E_R(\{0\})$.

\medskip
(3) $\Rightarrow$ (1): Assume that $O \leq C \leq A$, $O \leq C \leq B$. Since $C \leq A + B$, Lemma \ref{L2.2} ensures that there is a unique $O \leq D \leq I_{\mathcal{H}_{A,B}}$ such that $C = \Gamma(D)$. 
Then $D \leq \Gamma^{-1}(A) = R$ and $D \leq \Gamma^{-1}(B) = S$. Thus $RDR = SDS = O$ by assumption. Then $D^{1/2}R = D^{1/2}S = O$. Hence $D = O$ because $R + S = I_{\mathcal{H}_{A,B}}$.
\end{proof}

\begin{definition}\label{D3.3}
We say that $B = B_{1} + B_{2}$ with $B_{1}, B_{2} \leq B$ is an \emph{$A$-Lebesgue decomposition} of $B$ if $B_{1} \lessdot A$ and $B_{2} \perp A$.
\end{definition}

\begin{remark} \label{R3.4} {\rm 
Let us recall von Neumann's clever proof (see \cite[Theorem 6.10]{Rudin}) of Lebesgue decomposition for measures, which gave us a motivation to this work.  

Let $\mu,\nu$ be finite positive measures on a measurable space $\Omega$. We set $\mathcal{H} = L^{2}(\lambda)$ with $\lambda = \mu + \nu$. By Riesz's representation theorem we have $\mu(\Lambda) = (\mathbf{1}_\Lambda,R \mathbf{1}_\Omega) = (\mathbf{1}_\Lambda, g_{\mu})$, $\nu(\Lambda) = (\mathbf{1}_\Lambda, S \mathbf{1}_\Omega) = (\mathbf{1}_\Lambda, g_{\nu})$ for any measurable subset $\Lambda \subset \Omega$, where $g_{\mu}, g_{\nu}$ are the unique elements of $\mathcal{H} = L^{2}(\lambda)$ corresponding to the bounded linear functionals $f \in L^{2}(\lambda) \mapsto \int_\Omega f\,d\mu \in \mathbb{C}$, $f \in L^{2}(\lambda)\mapsto \int_\Omega f\,d\nu \in \mathbb{C}$, respectively, and $R, S$ are the multiplication operators of those functions, respectively. It is trivial that $R, S \geq O$ and $R+S = I$. We define $S_\mathrm{c} = S E_R((0,1])$ with the spectral (projection-valued) measure $E_R$ of $R$, and set $S_\mathrm{s} = S - S_\mathrm{c}$. Then we obtain that  $S_\mathrm{s}\mathbf{1}_\Omega = g_{\nu}\mathbf{1}_{\{g_{\mu} \neq 0\}}$ and $S_\mathrm{c}\mathbf{1}_\Omega = g_{\nu}\mathbf{1}_{\{g_{\mu} = 0\}}$. The $\mu$-Lebesgue decomposition of $\nu$ is given by $\nu = \nu_\mathrm{c} + \nu_\mathrm{s}$, where
$\nu_\mathrm{c}(\Lambda) = (\mathbf{1}_\Lambda,S_\mathrm{c} \mathbf{1}_\Omega)$ and $\nu_\mathrm{s}(\Lambda) = (\mathbf{1}_\Lambda,S_{s} \mathbf{1}_\Omega)$ for any measurable subset $\Lambda \subset \Omega$.}
\end{remark}

The above remark suggests us to define the absolutely continuous and the singular parts $(S_{A,B})_\mathrm{c}$ and $(S_{A,B})_\mathrm{s}$ of $S_{A,B}$ with respect to $R_{A,B}$ as follows.

\begin{lemma}\label{L3.5} 
We define $(S_{A,B})_\mathrm{c} \coloneqq S_{A,B}E_{R_{A,B}}((0,1])$ and $(S_{A,B})_\mathrm{s} = S_{A,B} - (S_{A,B})_\mathrm{c} = S_{A,B}E_{R_{A,B}}(\{0\})$ with the spectral (projection-valued) measure $E_{R_{A,B}}$ of $R_{A,B}$. Then $(S_{A,B})_\mathrm{c} \lessdot R_{A,B}$, $(S_{A,B})_\mathrm{s} \perp R_{A,B}$, and $S_{A,B} = (S_{A,B})_\mathrm{c} + (S_{A,B})_\mathrm{s}$ is an $R_{A,B}$-Lebesgue decomposition of $S_{A,B}$.
\end{lemma}
\begin{proof}
In the proof, we will write $R = R_{A,B}$ and $S=S_{A,B}$ for simplicity. 

We set $S_n = S E_R([1/n, 1]) \leq S_\mathrm{c}$. Then $S_n \nearrow S_\mathrm{c}$ as $n \to \infty$, and moreover,  
\[
S_n = S E_R([1/n, 1])\leq E_R([1/n,1]) \leq n R E_R([1/n, 1]) \leq nR.
\]
Hence $S_\mathrm{c} \lessdot R$.
    
Assume that $O \leq C \leq R$ and $O \leq C \leq S_{s}$. Then, 
\begin{gather*}
O \leq E_R(\{0\})CE_R(\{0\}) \leq R E_R(\{0\}) = O,\\
O \leq E_R((0,1])CE_R((0,1]) \leq S E_R((0,1]) = O. 
\end{gather*}
Then $C^{1/2}E_R(\{0\}) = O$. Similarly, we have $C^{1/2}E_R((0,1]) = O$. Hence $C = O$.
\end{proof}

\begin{proposition}\label{P3.6}
We set $B_\mathrm{c} = \Gamma_{A,B}((S_{A,B})_\mathrm{c})$ and $B_\mathrm{s} = \Gamma_{A,B}((S_{A,B})_\mathrm{s})$. Then $B = B_{c} + B_{s}$ is an $A$-Lebesgue decomposition of $B$.
\end{proposition}
\begin{proof}
Write $B_n = \Gamma_{A,B}(S_n)$ with $S_n$ as in the proof of Lemma \ref{L3.5}. By Lemma \ref{L2.2} and Remarks \ref{R2.3}(2) we have $B_n \leq B$ and $B_n \nearrow B$ as $n \to \infty$, and moreover, 
\[
B_n = \Gamma_{A,B}(S_n) \leq n \Gamma_{A,B}(R_{A,B}) = nA_{A,B}.
\]
Hence $B_\mathrm{c}$ is $A$-absolutely continuous.

Assume that $O \leq C \leq A$ and $O \leq C \leq B_\mathrm{s}$. Then, $O \leq \Gamma_{A,B}^{-1}(C) \leq \Gamma_{A,B}^{-1}(A) = R_{A,B}$ and $O \leq \Gamma_{A,B}^{-1}(C) \leq \Gamma_{A,B}^{-1}(B_\mathrm{s}) = S_\mathrm{s}$. Hence $\Gamma_{A,B}^{-1}(C) = O$, implying $C = O$ by Lemma \ref{L2.2}.
\end{proof}

\begin{proposition}\label{P3.7}
The operator $B_{c}$ is the maximal one among all the $A$-absolutely continuous positive bounded operators $C$ with $C \leq B$. 
\end{proposition}
\begin{proof}
    Let $C_n$ be an increasing sequence of positive bounded operators on $\mathcal{H}$ such that $C_n \nearrow C$ and $C_n \leq \alpha_n A$ for some $\alpha_n > 0$. Since $C_n \leq C \leq B \leq A + B$, we have
\[
\Gamma_{A,B}^{-1}(C_n) \leq \alpha_n \Gamma_{A,B}^{-1}(A) = \alpha_n R_{A,B}
\]
and $\Gamma_{A,B}^{-1}(C_n) \nearrow \Gamma_{A,B}^{-1}(C) \leq S_{A,B}$. Since $E_{R_{A,B}}(\{0\})\Gamma_{A,B}^{-1}(C_n)E_{R_{A,B}}(\{0\}) = O$, we have, as in the proof of Lemma \ref{L3.5}, $\Gamma_{A,B}^{-1}(C_n)E_{R_{A,B}}(\{0\}) = E_{R_{A,B}}(\{0\})\Gamma_{A,B}^{-1}(C_n) = O$. Taking the limit as $n \to \infty$, $\Gamma_{A,B}^{-1}(C)E_{R_{A,B}}(\{0\}) = E_{R_{A,B}}(\{0\})\Gamma_{A,B}^{-1}(C) = O$. Thus
\[
\Gamma_{A,B}^{-1}(C) 
= E_{R_{A,B}}((0,1])\Gamma_{A,B}^{-1}(C)E_{R_{A,B}}((0,1])
\leq E_{R_{A,B}}((0,1]) S_{A,B} E_{R_{A,B}}((0,1]) = (S_{A,B})_\mathrm{c}, 
\]
and hence $C \leq B_\mathrm{c}$.
\end{proof}

Here is the main but simple observation of these notes. 

\begin{theorem}\label{T3.8}
The operator $B_{c}$ coincides with Ando's $[A]B$; namely, $B_\mathrm{c} = \lim_{n \to \infty} (nA):B$ holds in the strong operator topology, where $(nA):B$ denotes the parallel sum of $nA$ and $B$ (see e.g., \cite[section 4]{FW71}).
Moreover, $[A]B = B_\mathrm{c}$ is given by the Pusz--Woronowicz functional calculus $\phi_\gtrdot(A,B)$ with function $\phi_\gtrdot(x,y) = \mathbf{1}_{(0,\infty)}(x)y$.
\end{theorem}

We remark that the function $\phi_\gtrdot(x,y)$ in the statement is indeed homogeneous (and trivially, bounded on any compact subset of $[0,\infty)^2$). In fact,     
\[
\phi_\gtrdot(\lambda x, \lambda y)
= 
\mathbf{1}_{(0,\infty)}(\lambda x)\lambda y
= 
\left\{
\begin{array}{ll}
\lambda y  &\quad (x > 0)\\
0 &\quad (x = 0)
\end{array}
\right\}
= 
\lambda \phi_\gtrdot(x,y), \quad (x,y) \in [0, \infty)^2,\  \lambda > 0.
\]

Although the previous proposition and \cite[Theorem 1]{Ando76} (both assert the maximality) immediately show $B_\mathrm{c} = [A]B$, we will give a self-contained proof,  i.e., without using Ando's result. 

\begin{proof} 
Since $E_{R_{A,B}}((0,1]) = \mathbf{1}_{(0,\infty)}(R_{A_,B})$, we have 
\[
B_\mathrm{c} = \Gamma_{A,B}(S_{A,B} E_{R_{A,B}}((0,1])) = \Gamma_{A,B}(\phi_\gtrdot(R_{A,B},S_{A,B})) = \phi_\gtrdot(A,B).
\]
    
For each $n$ we define $\phi_n(x,y) = \cfrac{nxy}{nx + y}$ with $\phi_n(0,0) = 0$. As remarked in \cite{HU21} (see the discussion just after Theorem 6 there) we have $(nA):B = \phi_n(A,B) = \Gamma_{A,B}(\phi_n(R_{A,B},S_{A,B}))$ for every $n$. Letting $f_n(y) \coloneqq \phi_n(1-y,y)$ we have
\[
\phi_n(R_{A,B},S_{A,B}) = f_n(S_{A,B}) = \int_{[0,1]}\frac{n(1 - y)y}{n(1 - y) + y}\,E_{S_{A,B}}(dy), 
\]
where $E_{S_{A,B}}$ denotes the spectral (projection-valued) measure of $S_{A,B}$. Since
\[
\left| 
\frac{n(1 - y)y}{n(1 - y) + y}
\right|
\leq 1
\quad 
\mathrm{and}
\quad
\frac{n(1 - y)y}{n(1 - y) + y}
\to \mathbf{1}_{(0,1]}(1 - y)y = \phi_\gtrdot(1-y,y), \quad \forall y \in [0,1],
\]
the bounded convergence theorem implies that
\[
\phi_n(R_{A,B},S_{A,B}) = f_n(S_{A,B}) \to \phi_\gtrdot(R_{A,B},S_{A,B}) = S_{A,B}E_{R_{A,B}}((0,1]) = S_\mathrm{c}
\]
in the strong operator topology. Hence we obtain $\phi_\gtrdot(A,B) = \lim_{n\to\infty} (nA):B$ in the strong operator topology too by Remarks \ref{R2.3}(2).
\end{proof}

We remark that Proposition \ref{P3.7} plus Theorem \ref{T3.8} imply \cite[Theorem 2]{Ando76}. 

The next corollary is immediate from Theorem \ref{T2.4} thanks to the above theorem. (We first learned from Hiai that Kosaki already confirmed the formula below many years ago.)

\begin{corollary}\label{C3.9} {\rm(H.\ Kosaki)} For each $k=1,2$, let $A_k,B_k$ be two positive bounded operators on a Hilbert space $\mathcal{H}_k$. Then 
\[
[A_1\otimes A_2](B_1\otimes B_2) = [A_1]B_1\otimes [A_2]B_2.
\]
\end{corollary}

The next proposition says that $[A]B$ admits a very similar form to \cite[Theorem 6]{Ko84}. 

\begin{proposition}\label{P3.10}
$[A]B$ is of the form 
\[
[A]B = B^{1/2} P_{A,B} B^{1/2}, 
\]
where $P_{A,B}$ is the orthogonal projection onto $\{\xi \in \mathcal{H}\:;\: Y_{A,B}Y_{A,B}^*\xi = \xi\}^{\perp}$ and coincides with $I - V_{A,B} V_{A,B}^* + V_{A,B} U_{A,B}^* U_{A,B} V_{A,B}^*$.
\end{proposition}
\begin{proof}
By Lemmas \ref{L2.6} and \ref{L2.7} we have
\[
B_\mathrm{c}
= 
B^{1/2} V_{A,B} \mathbf{1}_{(0,\infty)}(I - S_{A,B}) V_{A,B}^* B^{1/2}
= 
B^{1/2} \mathbf{1}_{(0,\infty)}(I - Y_{A,B}Y_{A,B}^*)B^{1/2}.
\]
Hence the desired assertion follows, since $\mathbf{1}_{(0,\infty)}(I - Y_{A,B}Y_{A,B}^*)$ is the orthogonal projection onto the closed subspace $\{\xi \in \mathcal{H}\:;\: Y_{A,B}Y_{A,B}^*\xi = \xi\}^{\perp}$.

Using Lemma \ref{L2.7} with $f(y) = \mathbf{1}_{(0,\infty)}(1 - y)$, we have
\[
\mathbf{1}_{(0,\infty)}(I - Y_{A,B}Y_{A,B}^*) 
= I - V_{A,B} V_{A,B}^* + V_{A,B} \mathbf{1}_{(0,\infty)}(R_{A,B})V_{A,B}^*.
\]
Since $\mathbf{1}_{(0,\infty)}(R_{A,B}) = U_{A,B}^* U_{A,B}$, we obtain the desired assertion. 
\end{proof}

One of the main results of \cite{Ando76} is Theorem 5 there, which can be regarded as an explicit description of the projection $P_{A,B}$ in terms of $A$ and $B$ (see \cite{Ko84} from the viewpoint of unbounded operators). We reconstruct its proof too.

\begin{lemma}\label{L3.11}
We have
\[
B^{1/2}(I - Y_{A,B}Y_{A,B}^{*})B^{1/2} = A^{1/2}(I - X_{A,B}X_{A,B}^*)A^{1/2}.
\]
\end{lemma}
\begin{proof}
We have
\begin{align*}
B^{1/2}Y_{A,B}Y_{A,B}^* B^{1/2}
&= 
T_{A,B}^* Y_{A,B}^* Y_{A,B}Y_{A,B}^* Y_{A,B} T_{A,B}\\
&= 
T_{A,B}^*(I_{\mathcal{H}_{A,B}} - X_{A,B}^* X_{A,B})(I_{\mathcal{H}_{A,B}} - X_{A,B}^* X_{A,B})T_{A,B} \quad \text{(bt Lemma \ref{L2.1})} \\
&= 
A + B 
- 2 T_{A,B}^* X_{A,B}^* X_{A,B} T_{A,B}
+ T_{A,B}^* X_{A,B}^* X_{A,B}X_{A,B}^* X_{A,B} T_{A,B} \\
&= 
B - A + A^{1/2}X_{A,B}X_{A,B}^* A^{1/2},
\end{align*}
implying the desired equation.
\end{proof}

\begin{remark}\label{R3.12} {\rm 
The above lemma can be understood as various expressions of the parallel sum. Namely, 
\begin{align*}
A:B = B^{1/2}(I - Y_{A,B}Y_{A,B}^*)B^{1/2} 
&= 
A^{1/2}(I - X_{A,B}X_{A,B}^*)A^{1/2} \\
&= 
B^{1/2}Y_{A,B}^* X_{A,B}A^{1/2} = A^{1/2}X_{A,B}^* Y_{A,B}B^{1/2}.
\end{align*}

The proof is as follows. As in the proof of \cite[Theorem 4.2]{FW71}, using a two-by-two matrix argument we can show that 
\begin{align*}
A^{1/2}(I - X_{A,B}^* X_{A,B}) 
= B^{1/2}Y_{A,B}^* X_{A,B},\quad A^{1/2} X_{A,B}Y_{A,B}^* 
= B^{1/2} (I - Y_{A,B}Y_{A,B}^*).
\end{align*}
On the other hand, since $A:B$ is given by the Pusz--Woronowicz functional calculus associated with $\phi(x,y) = xy/(x+y)$, we have
\[
A : B 
= 
\phi(A, B)
= T_{A,B}^* \phi(R_{A,B}, S_{A,B})T_{A,B}
= T_{A,B}^* R_{A,B}S_{A,B} T_{A,B}.
\]
Therefore, we have
\begin{align*}
R_{A,B}S_{A,B}
& = X_{A,B}^* X_{A,B}Y_{A,B}^* Y_{A,B} \\
&= S_{A,B} R_{A,B} = Y_{A,B}^* Y_{A,B} X_{A,B}^* X_{A,B} \\
&= R_{A,B}(I_{\mathcal{H}_{A,B}} - R_{A,B}) = X_{A,B}^*(I_{\mathcal{H}_{A,B}} - X_{A,B}X_{A,B}^*)X_{A,B} \\
&= (I - S_{A,B})S_{A,B} = Y_{A,B}^*(I - Y_{A,B}Y_{A,B}^*)Y_{A,B}.
\end{align*}
Hence we can obtain the above expressions of the parallel sum.}
\end{remark}

The next lemma is due to Izumino \cite[Lemma 2.3]{Iz89-1} (see also \cite[Theorem 4.2]{FW71}), but we will give its proof for the sake of completeness. 

\begin{lemma}\label{L3.13}
$\mathrm{ran}((I - Y_{A,B}Y_{A,B}^{*})^{1/2}) = \{\xi \in \mathcal{H}\:;\: B^{1/2}\xi \in \mathrm{ran}(A^{1/2})\}$.
\end{lemma}
\begin{proof}
Since $\mathrm{ran}(T) = \mathrm{ran}((TT^{*})^{1/2})$ for any bounded operator $T$, we have, by Lemma \ref{L3.11}, 
\begin{align*}
&\mathrm{ran}(B^{1/2}(I - Y_{A,B}Y_{A,B}^{*})^{1/2}) \\
&= \mathrm{ran}((B^{1/2}(I - Y_{A,B}Y_{A,B}^{*})B^{1/2})^{1/2}) \\
&= \mathrm{ran}((A^{1/2}(I - X_{A,B}X_{A,B}^{*})A^{1/2})^{1/2}) \\
&= \mathrm{ran}(A^{1/2}(I - X_{A,B}X_{A,B}^{*})^{1/2}) \subseteq \mathrm{ran}(A^{1/2}).
\end{align*}
Hence we have shown ($\subseteq$).

On the other hand, let $\xi \in \mathcal{H}$ be such that $B^{1/2}\xi \in \mathrm{ran}(A^{1/2})$. Then, $B^{1/2}\xi = A^{1/2} \eta$ for some $\eta \in \mathcal{H}$. Hence, $T_{A,B}^* Y_{A,B}^{*}\xi = T_{A,B}^* X_{A,B}^{*}\eta$ so that $Y_{A,B}(Y_{A,B}^{*}\xi - X_{A,B}^{*}\eta) = 0$ since $T_{A,B}$ has a dense range. Therefore, 
\[
\xi = (I - Y_{A,B}Y_{A,B}^*)\xi + Y_{A,B} X_{A,B}^*\eta 
\in (I - Y_{A,B}Y_{A,B}^*)\mathcal{H} + Y_{A,B} X_{A,B}^*\mathcal{H}.
\]
Since $X_{A,B}^{*}X_{A,B} + Y_{A,B}^{*}Y_{A,B} = R_{A,B} + S_{A,B} = I_{\mathcal{H}_{A,B}}$, we have 
\[
Y_{A,B}X_{A,B}^* X_{A,B}Y_{A,B}^* = Y_{A,B}(I - Y_{A,B}^* Y_{A,B})Y_{A,B}^* = (I - Y_{A,B}Y_{A,B}^*)Y_{A,B}Y_{A,B}^*.
\]
Hence we obtain that  
\begin{align*}
\mathrm{ran}(Y_{A,B}X_{A,B}^*) 
&= \mathrm{ran}((Y_{A,B}X_{A,B}^* X_{A,B}Y_{A,B}^*)^{1/2}) \\
&= \mathrm{ran}((I - Y_{A,B}Y_{A,B}^*)^{1/2}(Y_{A,B}Y_{A,B}^*)^{1/2})
\subseteq 
\mathrm{ran}((I - Y_{A,B}Y_{A,B}^*)^{1/2}).
\end{align*}
Thus we have shown ($\supseteq$).
\end{proof}

The next proposition includes \cite[Theorem 5]{Ando76}. 

\begin{proposition}\label{P3.14}
The range of $P_{A,B}$ is exactly the closure of $\{\xi \in \mathcal{H}\:;\: B^{1/2}\xi \in \mathrm{ran}(A^{1/2})\}$. In particular, the following are equivalent{\rm:}
\begin{itemize}
\item[(1)] $B \lessdot A$. 
\item[(2)] $B = [A]B$. 
\item[(3)] $P_{A,B} = I$. 
\end{itemize}
\end{proposition}
\begin{proof}
Note that $P_{A,B} = \mathbf{1}_{(0,\infty)}(I - Y_{A,B}Y_{A,B}^{*})$ is the orthogonal projection onto the $\ker(I - Y_{A,B}Y_{A,B}^{*})^{\perp}$. Since $\ker(I - Y_{A,B}Y_{A,B}^{*}) = \ker((I - Y_{A,B}Y_{A,B}^{*})^{1/2})$, we have
\begin{align*}
\ker(I - Y_{A,B}Y_{A,B}^{*})^{\perp}
&= 
\ker((I - Y_{A,B}Y_{A,B}^{*})^{1/2})^{\perp} \\
&= 
\overline{\mathrm{ran}(I - Y_{A,B}Y_{A,B}^{*})^{1/2})} \\
&= 
\overline{\{\xi \in \mathcal{H}\:;\: B^{1/2}\xi \in \mathrm{ran}(A^{1/2})\}}
\end{align*}
by Lemma \ref{L3.13}. 

\medskip
(1) $\Leftrightarrow$ (2): This follows from Proposition \ref{P3.7} ($[A]B$ is the maximal $A$-absolutely continuous part) plus Theorem \ref{T3.8}. 

(2) $\Rightarrow$ (3): Proposition \ref{P3.10} shows that $(I-P_{A,B})B^{1/2} = O$. Hence $(I-P_{A,B})\xi = 0$ for every $\xi \in \overline{\mathrm{ran}}(B^{1/2}) = \ker(B^{1/2})^\perp$. Thanks to the first part of assertion we see that $(I-P_{A,B})\eta = 0$ for all $\eta \in \ker(B^{1/2}) \subseteq \mathrm{ran}(P_{A,B})$. Hence $P_{A,B}=I$. 

(3) $\Rightarrow$ (2): This is trivial. 
\end{proof}

\begin{remark}\label{R3.15} {\rm 
Izumino \cite[Lemma 2.2]{Iz89-1} essentially showed $P_{A,B} = I - Y_{A,B}Y_{A,B}^* + Y_{A,B}U_{A,B}^* U_{A,B} Y_{A,B}^*$. We can also show $P_{A,B} = I - V_{A,B} V_{A,B}^* + V_{A,B} U_{A,B}^* U_{A,B} V_{A,B}^*$ by a direct calculation.

In fact, since $I - U_{A,B}^* U_{A,B}$ is the orthogonal projection onto $\ker(R_{A,B}) = \{\xi \in \mathcal{H}\:;\: R_{A,B}\xi = 0\} = \{\xi \in \mathcal{H}\:;\: S_{A,B}\xi = \xi\} = \{\xi \in \mathcal{H}\:;\: |Y_{A,B}|\xi = \xi\}$, we have $|Y_{A,B}|(I - U_{A,B}^* U_{A,B}) = I - U_{A,B}^* U_{A,B}$. Moreover, we have $|Y_{A,B}|(I - U_{A,B} U_{A,B}^*)|Y_{A,B}| = (|Y_{A,B}|(I - U_{A,B} U_{A,B}^*))(|Y_{A,B}|(I - U_{A,B} U_{A,B}^*))^*  = I - U_{A,B} U_{A,B}^*$. Therefore, 
\begin{align*}
Y_{A,B}Y_{A,B}^* - Y_{A,B}U_{A,B}^* U_{A,B} Y_{A,B}^* 
&= V_{A,B} |Y_{A,B}|(I - U_{A,B}^* U_{A,B})|Y_{A,B}| V_{A,B}^* \\
&= V_{A,B} (I - U_{A,B}^* U_{A,B})V_{A,B}^*.
\end{align*}
}
\end{remark}

Ando \cite[section 6]{Ando76} also gave a characterization of uniqueness of $A$-Lebesgue decomposition. His proof can be regarded as an application of Proposition \ref{P3.14} and a theorem of von Neumann (see \cite[Theorem 3.6]{FW71}). Hence we have successfully reconstructed Ando's theory as well as Kosaki's description of $[A]B$ \cite[Theorem 6]{Ko84}.

\section{Radon--Nikodym derivatives}

Throughout this section, we will assume that $A$ is non-singular. Since $A \leq (A+B)$, $\mathcal{H}_{A,B} = \mathcal{H}$ and $T_{A,B} = (A+B)^{1/2}$ hold under the assumption. 

\medskip
The first lemma is essentially due to Kosaki \cite[Lemma 3]{Ko84}. 

\begin{lemma}\label{L4.1}
Let $B = B_\mathrm{c} + B_\mathrm{s}$ be the $A$-Lebesgue decomposition of $B$ as in Proposition \ref{P3.6}. Then the quadratic form $q$ on the Hilbert space $\mathcal{H}$ with domain $\mathcal{D}(q) = \mathrm{ran}(A^{1/2})$ defined by
\[
q(A^{1/2}\xi) = \|B_\mathrm{c}^{1/2}\xi\|^{2}, \quad \xi \in \mathcal{H}
\]
becomes a closable positive quadratic form.
\end{lemma}
\begin{proof}
By the $A$-absolute continuity of $B_\mathrm{c}$, $A^{1/2}\xi = 0$ implies $B_\mathrm{c}^{1/2}\xi$, that is, the above $q$ is well-defined. Then the proof of (a) $\Longrightarrow$ (b) of \cite[Lemma 3]{Ko84} shows that $q$ is a closable.
\end{proof}

Let $\bar{q}$ be the closure of the above $q$, and there is a unique positive self-adjoint operator $D_{A,B}$ such that $\mathrm{dom}(D_{A,B}^{1/2}) = \mathcal{D}(\bar{q})$, the domain of $\bar{q}$, and $\bar{q}(\xi) = \| D_{A,B}^{1/2}\xi\|^2$ for $\xi \in \mathcal{D}(\bar{q})$. It is a standard fact that $\mathcal{D}(q)$ is a core of $D_{A,B}^{1/2}$.  

\begin{proposition}\label{P4.2}
$Z \coloneqq D_{A,B}^{1/2}A^{1/2}$ is a bounded operator on the Hilbert space $\mathcal{H}$ and $Z^* Z = B_\mathrm{c}$.
\end{proposition}
\begin{proof}
Since $\mathrm{ran}(A^{1/2}) \subseteq \mathcal{D}(D_{A,B}^{1/2})$, $Z$ is defined on the whole $\mathcal{H}$ and also closed thanks to the closed graph theorem. Moreover, $(Z^* Z\xi,\xi) = \|Z\xi\|^2 = \|B_\mathrm{c}^{1/2}\xi\|^2 = (B_\mathrm{c}\xi,\xi)$ for all $\xi \in \mathcal{H}$, and thus $Z^* Z = B_\mathrm{c}$.
\end{proof}

The above proposition may be understood as an extension of the Douglas decomposition theorem (see e.g., \cite[Theorem 2.1]{FW71}), and the operator $D_{A,B}$ may be understood to play a role of Radon-Nikodym derivative.

\medskip
We define $h(x) = \frac{1-x}{x}\mathbf{1}_{(0,1]}(x)$and $
h_n(x) = \frac{1-x}{x}\mathbf{1}_{[1/n,1]}(x)$ ($n=1,2,\dots$) for all $x \in [0,1]$.

\begin{lemma}\label{L4.3}
\[
q(A^{1/2}\xi) 
= \|B_{c}^{1/2}\xi\|^2
= \int_0^1 h(x)\,\| E_{XX^*}(dx)A^{1/2}\xi\|^2,
\]
where $E_{XX^*}$ is the spectral projection of $X_{A,B}X_{A,B}^*$. 
\end{lemma}   
\begin{proof}
Since $h_n(x)$ is a bounded Borel function on $[0,1]$ and $x\,h_n(x) \nearrow \phi_\gtrdot(x,y) = y\mathbf{1}_{(0,\infty)}(x)$ on $\{ (x,y) \in [0,\infty)^2\,;\, x + y = 1\}$ as $n\to\infty$, we have $R_{A,B}\,h_n(R_{A,B}) \nearrow \phi_\gtrdot(R_{A,B},S_{A,B})$, and hence, $\Gamma_{A,B}(R_{A,B}\, h_n(R_{A,B})) \nearrow \phi_\gtrdot(A,B) = B_\mathrm{c}$ as $n\to\infty$. We have, for any $\xi \in \mathcal{H}$, 
\begin{align*}
q(A^{1/2}\xi) 
= (B_\mathrm{c}\xi,\xi)_{\mathcal{H}}
&= 
\lim_{n \to \infty}(\Gamma_{A,B}(R_{A,B}\,h_n(R_{A,B}))\xi,\xi) \\
&= 
\lim_{n \to \infty}(\Gamma_{A,B}(R_{A,B}^{1/2}\,h_n(R_{A,B})R_{A,B}^{1/2})\xi,\xi) \\
&= 
\lim_{n \to \infty} (A^{1/2}(U_{A,B} h_n(R_{A,B}) U_{A,B}^*)A^{1/2}\xi,\xi) \\
&= 
\lim_{n \to \infty} (A^{1/2}(h_n(X_{A,B}X_{A,B}^*))A^{1/2}\xi,\xi) \\
&= 
\lim_{n \to \infty} 
\int_0^1 h_n(x)\,\| E_{XX^*}(dx)A^{1/2}\xi\|^2 \\
&= 
\int_0^1 h(x)\,\| E_{XX^*}(dx)A^{1/2}\xi\|^2
\end{align*}
by the monotone convergence theorem.
\end{proof}

The above proof also shows the next corollary. The last part uses \cite[Lemma 2.5]{HUW22}. 

\begin{corollary}\label{C4.4}
The sequence $h_n(X_{A,B}X_{A,B}^{*})$ is monotone increasing and converges to an element of the extended positive part $\widehat{B(\mathcal{H})}_+$ denoted by $h(X_{A,B}X_{A,B}^*)$. Then we have  
\[
B_\mathrm{c} = Z^* Z = \lim_{n \to \infty}A^{1/2}h_n(X_{A,B}X_{A,B}^*)A^{1/2} = A^{1/2}h(X_{A,B}X_{A,B}^*)A^{1/2},
\]
where the second equality is in the strong operator topology and the third equality is justified as elements of $\widehat{B(\mathcal{H})}_+$. 
\end{corollary}

It is natural to regard $h(X_{A,B}X_{A,B}^*) \in \widehat{B(\mathcal{H})}_+$ as a kind of Radon--Nikodym derivative for $B_\mathrm{c}$ with respect to $A$. We set the lower semi-continuous quadratic form $p$ by
\[
p(\xi) = \int_0^1 h(x)\,\| E_{XX^*}(dx)\xi\|^2
\]
with domain $\{\xi \in \mathcal{H}\:;\: \int_{0}^{1}h(x)\left\| E_{XX^*}(dx)\xi\right\|^2 < \infty\}$. Then, $p$ agrees with $\overline{q}$ on the $\mathrm{ran}(A^{1/2})$. Moreover, we have
\[
\mathcal{D}(\overline{q}) \subseteq \mathcal{D}(p)
\quad \mathrm{and} \quad 
p(\xi) \leq \overline{q}(\xi).
\]
In fact, for every $\xi \in \mathcal{D}(\overline{q}) = \mathcal{D}(D_{A,B}^{1/2})$ there exists a sequence $\{\xi_n\}$ in $\mathrm{ran}(A^{1/2})$ such that $\xi_n \to \xi$ and $D_{A,B}^{1/2}\xi_n \to D_{A,B}^{1/2}\xi$. Then we have
\[
p(\xi) 
\leq \liminf_{n \to \infty} p(\xi_n) 
= \liminf_{n \to \infty} q(\xi_n)
= \liminf_{n \to \infty} \|D_{A,B}^{1/2}\xi_n\|^2
= \|D_{A,B}^{1/2}\xi\|^2.
\]

\begin{question}
When does $\overline{q} = p$ hold ?
\end{question}

\bigskip
We have known that $[A]B$ is a special example of PW-functional calculus. Thus, we generalize the above discussion to a general operator given by PW-functional calculus. The consequence can be regarded as an attempt to write a PW-functional calculus operator to be of the form $A^{1/2}(\cdots)A^{1/2}$ like Kubo--Ando's form of operator means/connections. 

\begin{proposition}\label{P4.6}
Let $\phi(x,y)$ be a non-negative homogeneous Borel function on $[0,\infty)^2$ that is bounded on any compact subset. Assume that $\phi(x,y) = 0$ whenever $x=0$. Then, the PW-functional calculus $\phi(A,B)$ is $A$-absolutely continuous and of the form
\[
\phi(A,B) = Z_\phi^* Z_\phi \quad \text{with} \quad Z_\phi = T_\phi^{1/2}A^{1/2}, 
\]
where $T_\phi$ is a unique positive self-adjoint operator corresponding to the  closure of the closable positive quadratic form  
\[
q_\phi(A^{1/2}\xi) = \|\phi(A,B)^{1/2}\xi\|^2, \quad \xi \in \mathcal{H}
\]
with domain $\mathrm{ran}(A^{1/2})$.
\end{proposition}
\begin{proof}
Let $(h_\phi)_n(x) = \frac{\phi(x,1-x)}{x}\mathbf{1}_{[1/n, 1]}(x)$ on $x \in [0, 1]$. Then $x\,(h_\phi)_n(x) \nearrow \phi(x,1-x)\mathbf{1}_{(0,1]}(x) = \phi(x,y)$ by assumption. We have $R_{A,B}\,(h_\phi)_n(R_{A,B}) \nearrow \phi(R_{A,B},S_{A,B})$. Moreover, $R_{A,B}\,(h_\phi)_n(R_{A,B})\leq \alpha_n R_{A,B}$ with $\alpha_n \coloneqq n\sup\{\phi(x,1-x)\,;\,x\in [0,1]\} < \infty$ 
implies that $\Gamma_{A,B}(R_{A,B}\,(h_\phi)_n(R_{A,B})) \leq \alpha_n A$ and 
\[
\Gamma_{A,B}(R_{A,B}\,(h_\phi)_n(R_{A,B})) \nearrow \Gamma_{A,B}(\phi(R_{A,B},S_{A,B})) = \phi(A, B).
\] 
Hence $\phi(A,B) \lessdot A$. Then we can apply the above discussion to $\phi(A,B)$ instead of $B_\mathrm{c}$.
\end{proof}

With $(h_\phi)_n(x) = \frac{\phi(x,1-x)}{x}\mathbf{1}_{[1/n, 1]}(x) \nearrow \frac{\phi(x,1-x)}{x}\,\mathbf{1}_{(0,1]}(x) = h_\phi(x)$ for all $x \in [0, 1]$, the above proof shows that the sequence $(h_\phi)_n(X_{A,B}X_{A,B}^*)$ is monotone increasing and we have
\[
\phi(A,B) = \lim_{n\to\infty} A^{1/2}\,(h_\phi)_n(X_{A,B}X_{A,B}^*)\,A^{1/2} = A^{1/2}\, h_\phi(X_{A,B}X_{A,B}^*)\,A^{1/2}, 
\]
where the first equation is in the strong operator topology and the second equation is understood as elements of the extended positive part $\widehat{B(\mathcal{H})}_+$.

In closing of this section, we should mention that the consideration here is related to Kosaki's works \cite{Ko84},\cite[section 4]{Ko18} and \cite{Ko}, where more detailed analyses from the viewpoint of unbounded operators are carried out. We regard Kosaki's unbounded operator approach to Lebesgue decomposition as a kind of attempt to develop Radon--Nikodym theorem for positive operators.  

\appendix
\section{A detailed account of Remark \ref{R2.5}}

We will give a detailed account about Remark \ref{R2.5}. Let us consider $\phi_\alpha(x,y)$ with $\alpha>1$ and $\psi(x,y)$ defined by 
\[
\phi_\alpha(x,y) 
= 
\begin{cases}
0 & (x=y=0), \\
\infty & (x>0=y), \\
(x/y)^\alpha y & (\text{otherwise}), 
\end{cases}
\qquad 
\psi(x,y)
=
\begin{cases}
0 & (x=y=0), \\
\infty & (x>0=y), \\
x\log(x/y) & (\text{otherwise}). 
\end{cases}
\]
Then it is easy to confirm that 
\begin{equation}\label{EqA1}
\phi_\alpha(x_1 x_2, y_1 y_2) = \phi_\alpha(x_1,y_1)\, \phi_\alpha(x_2,y_2), \qquad 
\psi(x_1 x_2, y_1 y_2) = \psi(x_1,y_1)x_2 + x_1\psi(x_2,y_2)
\end{equation}
hold for all $(x_1,y_1), (x_2,y_2) \in [0,\infty)^2$. 

For any $(A_i,B_i) \in B(\mathcal{H}_i)_+\times B(\mathcal{H}_i)_+$, $i=1,2$, we have $\phi_\alpha(A_1\otimes A_2, B_1\otimes B_2), \psi(A_1\otimes A_2, B_1\otimes B_2) \in \widehat{B(\mathcal{H}_1\otimes\mathcal{H}_2)}_\mathrm{lb}$ by PW-functional calculus developed in \cite{HUW22}. 

\begin{proposition}\label{PA1} For any $\rho_i \in B(\mathcal{H}_i)_*^+$, $i=1,2$, we have 
\begin{equation}\label{EqA2}
\begin{aligned}
\phi_\alpha(A_1\otimes A_2, B_1\otimes B_2)(\rho_1\otimes\rho_2) 
&=
\phi_\alpha(A_1,B_1)(\rho_1)\,\phi_\alpha(A_2,B_2)(\rho_2), \\
\psi(A_1\otimes A_2, B_1\otimes B_2)(\rho_1\otimes\rho_2) 
&=
\psi(A_1,B_1)(\rho_1)\,\rho_2(A_2) + \rho_1(A_1)\,\psi(A_2,B_2)(\rho_2), 
\end{aligned}
\end{equation}
where $\rho_1\otimes\rho_2$ is an element of $B(\mathcal{H}_1\otimes\mathcal{H}_2)_*^+$ via $B(\mathcal{H}_1\otimes\mathcal{H}_2) \cong B(\mathcal{H}_1)\otimes B(\mathcal{H}_2)$ as von Neumann algebra tensor product. 
\end{proposition}
\begin{proof}
The proof is essentially the same as that of Theorem \ref{T2.4}, and we will prove only the second formula. 

We will use the notation in the proof of Theorem \ref{T2.4}. We have
\begin{align*}
&\psi (A_1\otimes A_2, B_1\otimes B_2)(\rho_1\otimes\rho_2) \\
&=
\int_0^1 \int_0^1 \psi(x_1 x_2, (1-x_1)(1-x_2))\,(T(\rho_1\otimes\rho_2)T^*)(E_1(dx_1)\otimes E_2(dx_2)) \\
&=
\int_0^1 \int_0^1 \psi(x_1 x_2, (1-x_1)(1-x_2))\,(T_{A_1,B_1}\rho_1 T_{A_1,B_1}^*)(E_1(dx_1))\,(T_{A_2,B_2}\rho_2 T_{A_2,B_2}^*)(E_2(dx_2)) \\
&=
\int_0^1\psi(x_1,1-x_1)\, (T_{A_1,B_1}\rho_1 T_{A_1,B_1}^*)(E_1(dx_1))\,\int_0^1 x_2\,(T_{A_2,B_2}\rho_2 T_{A_2,B_2}^*)(E_2(dx_2)) \\
&\qquad+ 
\int_0^1 x_1\,(T_{A_1,B_1}\rho_1 T_{A_1,B_1}^*)(E_1(dx_1))\, \int_0^1\psi(x_2,1-x_2)\, (T_{A_2,B_2}\rho_2 T_{A_2,B_2}^*)(E_2(dx_2)) 
\end{align*}
by \eqref{EqA1} and the Fubini theorem. This is nothing but the right-hand side of the desired formula.
\end{proof}

Assume further that the above $A_i, B_i$, $i=1,2$, are of trace class. As in \cite[Proposition 7.11]{HUW22} together with (the proof of) \cite[Proposition 1.10]{Haagerup79-1} we have 
\begin{align*}
\mathrm{Tr}(\phi_\alpha(A_1\otimes A_2,B_1\otimes B_2))
&=
\mathrm{Tr}(\phi_\alpha(A_1,B_1))\,\mathrm{Tr}(\phi_\alpha(A_2,B_2)), \\
\mathrm{Tr}(\psi(A_1\otimes A_2,B_1\otimes B_2)) 
&=
\mathrm{Tr}(\psi(A_1,B_1))\,\mathrm{Tr}(A_2) + \mathrm{Tr}(A_1)\,\mathrm{Tr}(\psi(A_2,B_2)).
\end{align*}
These formulas are probably well known among specialists of quantum information theory, and suggest that PW-functional calculus is much useful to treat some kinds of binary operations for positive operators. 

In closing of this appendix, we would like to ask whether or not $\phi_\alpha(A_1\otimes B_1,A_2\otimes B_2)$ and $\psi(A_1\otimes B_1, A_2\otimes B_2)$ are uniquely determined by \eqref{EqA2}. We believe that this question allows an affirmative answer from the viewpoint of tensor products of operator valued weights, but we have no proof about it at the present moment. 

\section{Continuity of PW-functional calculus}

We will reformulate the assertion of Theorem \ref{T3.8} in an abstract fashion. Let $\phi_n(x,y)$ be a sequence of extended real-valued homogeneous Borel functions over $[0,\infty)^2$, each of which is bounded from below on any compact subsets, and let $\phi(x,y)$ be an extended real-valued homogeneous Borel functions over $[0,\infty)^2$, which is bounded from below on any compact subsets. Assume that 
\[
\lim_{n\to\infty}\int_0^1 \phi_n(x,1-x)\,\nu(dx) = \int_0^1 \phi(x,1-x)\,\nu(dx) \quad \text{(possibly $+\infty=+\infty$)}
\]
for all finite Borel measures $\nu$ on $[0,1]$. Then we have 
\begin{align*}
\phi_n(A,B)(\rho) 
&= 
\int_0^1 \phi_n(x,1-x)\,(T_{A,B}\rho T_{A,B}^*)(E_{R_{A,B}}(dx)) \\
&\to 
\int_0^1 \phi(x,1-x)\,(T_{A,B}\rho T_{A,B}^*)(E_{R_{A,B}}(dx)) 
=
\phi(A,B)(\rho)
\end{align*}
as $n\to\infty$ for all $\rho \in B(\mathcal{H})_*^+$, where $E_{R_{A,B}}$ denotes the spectral (projection-valued) measure of $R_{A,B}$. 

Now, we further suppose that all the $\phi_n(x,y)$ and $\phi(x,y)$ are bounded on any compact subsets. In this case, all the $\phi_n(A,B)$ and $\phi(A,B)$ fall into $B(\mathcal{H})$. Assume here that $\phi_n(x,y)\to\phi(x,y)$ pointwisely as $n\to \infty$ and that
\[
\sup_n \sup_{x \in [0,1]} |\phi_n(x,1-x)| < +\infty,
\]
or equivalently, the $\phi_n(x,y)$ are bounded uniformly in $n$ on each compact subset thanks to the homogeneity. Then we have, by the bounded convergence theorem,  
\begin{align*}
&\Vert(\phi_n(R_{A,B},S_{A,B})-\phi(R_{A,B},S_{A,B}))\xi\Vert_{\mathcal{H}_{A,B}}^2 \\
&=
\int_0^1 |\phi_n(x,1-x)-\phi(x,1-x)|^2\,\Vert E_{R_{A,B}}(dx)\xi\Vert^2 _{\mathcal{H}_{A,B}} \to 0
\end{align*}
as $n\to\infty$ for all $\xi \in \mathcal{H}_{A,B}$. By Remarks \ref{R2.3}(2) we obtain that $\phi_n(A,B)=\Gamma_{A,B}(\phi_n(R_{A,B},S_{A,B})) \to \Gamma_{A,B}(\phi(R_{A,B},S_{A,B})) = \phi(A,B)$ in the strong operator topology as $n\to\infty$. This is indeed an abstraction of Theorem \ref{T3.8}.

\end{document}